\documentclass[titlepage]{article}

\usepackage{rp}

\begin{document}

\begin{titlepage}
	{\Huge Random Projections\\ and Dimension Reduction \par}
	\vspace{1cm}
	{\Large Rishi Advani --- Cornell University \\
	Madison Crim --- Salisbury University \\
	Sean O'Hagan --- University of Connecticut\par}
	\vspace{1.5cm}
	{\large\bfseries Summer@ICERM 2020\par}
	\vspace{2cm}
	\hfill\includegraphics[width=0.4\textwidth]{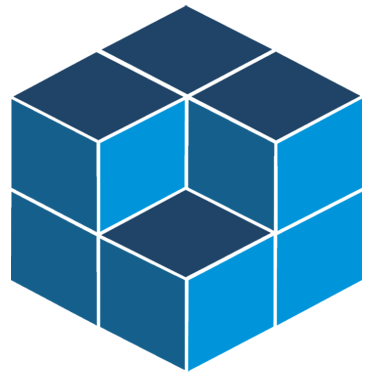}\par

	\vfill

	{\normalsize Thank you to ICERM for (virtually) hosting us this summer, and thank you to all the staff for making this program possible. Thank you to our organizers, Akil Narayan and Yanlai Chen, along with our TAs, Justin Baker and Liu Yang, for supporting us throughout this program.\par}
\end{titlepage}

\tableofcontents
\newpage

\setlength{\parskip}{1em}

\section{Introduction}
This paper, broadly speaking, covers the use of randomness in two main areas: low-rank approximation and kernel methods.

\subsection{Low-rank Approximation}
Low-rank approximation is very important in numerical linear algebra. Many applications depend on matrix decomposition algorithms that provide accurate low-rank representations of data. In modern problems, however, various factors make this hard to accomplish:
\begin{itemize}
    \item the amount of data and amount of features is absurdly large at times
    \item we often have missing or inaccurate data
    \item it may not be possible to simultaneously store all the data in memory
\end{itemize}

One solution to these problems is the use of random projections. Instead of directly computing the matrix factorization, we randomly project the matrix onto a lower-dimensional subspace and then compute the factorization. Often, we are able to do this without significant loss of accuracy.

We describe how randomization can be used to create more efficient algorithms to perform low-rank matrix approximation, as well as introducing a novel randomized algorithm for matrix decomposition. Compared to standard approaches, random algorithms are often faster and more robust. With these randomized algorithms, analyzing massive data sets becomes tractable.

\subsection{Kernel Methods}
Kernel methods are almost diametrically opposite from low-rank approximation. The idea is to project low-dimensional data into a higher-dimensional `feature space,' such that it is linear separable in the feature space. This enables the model to learn a nonlinear separation of the data.

As before, with large data matrices, computing the kernel matrix can be expensive, so we use randomized methods to approximate the matrix.

In addition, we propose an extension of the random Fourier features kernel in which hyperparameter values are randomly sampled from an interval or Borel set.

The experiments discussed in this paper can be found on our GitHub repository and website using the following links:
\begin{itemize}
    \item \url{https://github.com/rishi1999/random-projections}
    \item \url{https://rishi1999.github.io/random-projections/}
\end{itemize}

\section{Johnson-Lindenstrauss Lemma}
The Johnson-Lindenstrauss Lemma, first appearing in \cite{jlpaper}, is a fundamental result in this area and falls under the umbrella of \emph{concentration of measure}. 

Simply put, the Johnson-Lindenstrauss Lemma describes the existence of a map from a higher dimensional space $\mathbb{R}^d$ into a lower dimensional space $\mathbb{R}^k$ that preserves pairwise distances between the $n$ points up to an error tolerance $0<\varepsilon<1$, with $k$ on the order of $\varepsilon^{-2}\log n$. 

In applications with which we are concerned, the data (collection of points) can be viewed as a matrix $A\in \Reals^{n\times d}$, with each row representing a point in $\mathbb{R}^d$, and the map in question can be represented by a matrix in $\Reals^{k\times d}$.

\begin{lemma}[Johnson-Lindenstrauss]
Let $\{x_1,...,x_n\}$ be a collection of data points in $\mathbb{R}^d$. Let $k\in\mathbb{N}$ such that 
\[
k > C\cdot \frac{\log n}{\varepsilon^2} \quad\quad (C\approx 24)
\]
Then there exists a linear map $f: \Reals^d\to\Reals^k$ such that for any $x_i,x_j\in X$,
\[
(1-\varepsilon)\,\|x_i-x_j\|_2^2 \leq \|f(x_i)-f(x_j)\|_2^2 \leq (1+\varepsilon)\,\|x_i-x_j\|_2^2
\]
\end{lemma}

\begin{remark}
The proof we give is probabilistic. Reconstructing the proof from \cite{jl_lemma}, we will take a random rectangular matrix with entries drawn from a standard normal distribution, first show that the expectation of the squared 2-norm of the low-dimensional projection of an arbitrary vector in $\mathbb{R}^n$ is the equivalent to its original squared 2-norm in higher dimensional space, and then show that we can be within an arbitrary tolerance with positive probability.
\end{remark}

\begin{proof}
  Let $u\in\mathbb{R}^d$ and let $R\in\Reals^{k\times d}$, where every entry in $R$ is drawn i.i.d. from a standard normal distribution. Set $v=\frac{1}{\sqrt{k}}Ru$. Here the coefficient $\frac{1}{\sqrt{k}}$ represents a normalization factor.
\begin{prop}
$\mathbb{E}[\|v\|_2^2] = \|u\|_2^2$
\end{prop}
\begin{proof}
\begin{align*}
    \mathbb{E}[\|v\|_2^2] &= \mathbb{E}\Big[\sum_{i=1}^kv_i^2\Big] \\
    &= \sum_{i=1}^k \mathbb{E} \,[v_i^2] \\
    &= \sum_{i=1}^k \frac{1}{k}\mathbb{E} \,[(\sum_j R_{ij}u_j)^2] \\
    &= \sum_{i=1}^k \frac{1}{k} \sum_{\substack{1\leq j, l\leq d}} u_ju_l\mathbb{E}[R_{ij}R_{il}] \\
    &= \sum_{i=1}^k \frac{1}{k}\sum_{\substack{1\leq j, l\leq d}} u_ju_l\delta_{jl} \\
    &= \sum_{i=1}^k \frac{1}{k}\sum_{j=1}^d u_j^2 = \sum_{j=1}^d u_j^2 = \|u\|_2^2 \\
\end{align*}
\end{proof}

Now, we have determined the mean of our random variable $\|v\|_2^2$, and it remains to show that its value concentrates around this mean. More specifically, we want to put an upper bound on the probability that we are arbitrarily far from the mean, and later to bound the probability of the union of all of these events to reach the desired conclusion.

\begin{prop}
$\Pr(\|v\|_2^2 \geq (1+\varepsilon)\|u\|_2^2) \leq n^{-2}$
\end{prop}
\begin{proof}
  We define a random variable $X\in\Reals^k$ as a scaled version of $v$, such that $X=\frac{\sqrt{k}}{\|u\|}v$. Thus, each element $x_i=\frac{1}{\|u\|}R_i^T \, u$ for $i=1\ldots k$. Additionally, denote $x=\|X\|_2^2 = \sum_{i=1}^k x_i^2 = \frac{k\|v\|_2^2}{\|u\|_2^2}$. Since $v_i\ = \frac{1}{\sqrt{k}}R_i^T \, u_i$, we have $v_i\sim N(0, \frac{\|u\|_2^2}{k})$, and thus $x_i\sim N(0,1)$.
  
First, substitute to obtain
\[
    \Pr(\|v\|_2^2 \geq (1+\varepsilon)\|u\|_2^2) = \Pr(x\geq (1+\varepsilon)k). 
\]
Exponentiating both sides and multiplying by $e^\lambda$ for any arbitrary real $\lambda$ yields
\[
\Pr(e^{\lambda x} \geq e^{\lambda(1+\varepsilon)k}).
\]
Next, we use Markov's inequality, which states that for a nonnegative random variable $X$, we have $\Pr(X\geq a) \leq \frac{\E[X]}{a}$, in order to get the upper bound
\[
\Pr(e^{\lambda x} \geq e^{\lambda(1+\varepsilon)k}) \leq \frac{\E [e^{\lambda x}]}{{e^{\lambda(1+\varepsilon)k}}}.
\]
Since $x_i$, and thus $x_i^2$, is independent, the expectation of the product equates to the product of the expectation, which yields equality with the product
\[
    \prod_{i=1}^k\frac{\E[e^{\lambda x_i^2}]}{e^{\lambda(1+\varepsilon)k}} \,.
\]
Since $x_i$, and thus $x_i^2$, is identically distributed, we obtain the final upper bound:
\[
\frac{\big(\E[e^{\lambda x_i^2}]\big)^k}{e^{\lambda (1+\varepsilon)k}} \,.
\]

To evaluate the expectation in the numerator, note that, since $x_i \sim N(0,1)$, we have $x_i^2 \sim \chi_1^2$. We now use the moment generating function from mathematical statistics: observe that if $X\sim\chi_1^2$, we have $M_X(t)=\E[e^{tX}]=(1-2t)^{-1/2}$.

Thus, this yields
\[
\ldots = \Big(\frac{1}{\sqrt{1-2\lambda}\cdot e^{\lambda(1+\varepsilon)}}\Big)^k \,,
\]
and since this is true for any arbitrary $0<\lambda<\frac{1}{2}$, we may choose $\lambda=\frac{\varepsilon}{2(1+\varepsilon)}$, and obtain
\[
\ldots = [(1+\varepsilon)e^{-\varepsilon}]^{k/2}.
\]
For the next step, we use an inequality built on the Taylor expansion of $\log(1+a)$.

\begin{lemma}
For a positive real $a$, 
\[
\log(1+a) \leq a - \frac{a^2}{2}+\frac{a^3}{3}\,.
\]
\end{lemma}
\begin{proof}
Let $f(a) = \exp(a-\frac{a^2}{2}+\frac{a^3}{3})-(1+a)$. Taking the derivative, we obtain $f'(a)=(a^2-a+1)\exp(a-\frac{a^2}{2}+\frac{a^3}{3})-1$. This derivative is always positive, which can be verified by taking its derivative: $f''(a)=\exp(1/6a(6-3a+2a^2))a^2(3-2a+a^2)$, which is always positive on $a>0$ as the exponential and the two polynomial factors are all strictly positive on $a>0$. Since we know $f'(0)=0$, and $f''(a)>0$ for $a>0$, this means $f'(a)>0$ for $a>0$, and thus since $f(0)=0$, we know $f(a)>0$ for $a>0$. Thus, $\exp(a-\frac{a^2}{2}+\frac{a^3}{3}) > 1+a$ for $a>0$. Taking the logarithm of both sides yields the desired result.
\end{proof}

Using this lemma, we achieve the upper bound
\[
  \ldots \leq \exp\Big(-(\frac{\varepsilon^2}{2}-\frac{\varepsilon^3}{3})\frac{k}{2}\Big) \leq e^{-2\log n} \leq n^{-2} \,,
\]
where the first inequality comes from our bound on $k$.
\end{proof}

We can apply a similar procedure to obtain the bound
\[
\Pr(\|v\|_2^2 \geq (1-\varepsilon)\|u\|_2^2)\leq n^{-2} \,,
\]
and we may combine these using the subadditivity of probability (the probability of of a union of events is less than or equal to the sum of their probabilities) to yield 
\[
\Pr\Bigg(\|v\|_2^2 \not\in \Big((1-\varepsilon)\|u\|_2^2,(1+\varepsilon)\|u\|_2^2\Big)\Bigg)\leq 2n^{-2} \,.
\]
Now, since $u$ is an arbitrary vector in $\Reals^d$, we may let $u=x_i-x_j$ for $x_i,x_j$, $i,j\leq n$, and define the event
\[
E_{ij} := \|f(x_i)-f(x_j)\|_2^2 \not\in \Big((1-\varepsilon)\|x_i-x_j\|_2^2, (1+\varepsilon)\|x_i-x_j\|_2^2\Big) \,.
\]

We then obtain the union bound
\[
  \Pr\Big(\bigcup_{\substack{i\leq n\\j<i}} E_{ij} \Big) \leq \sum_{\substack{i\leq n\\j<i}} \Pr(E_{ij}) \leq \frac{n(n-1)}{2}\cdot 2n^{-2}=1-\frac{1}{n} \,.
\]
  Thus, the probability that all of the pairwise distances fall within the desired intervals is given by the complement, and we obtain a lower bound of $\frac{1}{n}$. Since the probability of the event occurring is greater than 0, there must exist a map that satisfies the restrictions we require, concluding the proof.
\end{proof}


\section{Low-rank Approximation}
\subsection{Singular Value Decomposition}

\subsubsection{Deterministic SVD} \label{section: d_svd}
Given any matrix $A\in \mathbb{R}^{m\times n}$, we can express $A$ using the singular value decomposition:
\begin{equation}
\label{detsvd}
A = U_{m\times m}\Sigma_{m\times n} V^*_{n\times n}
\end{equation}
where $U$ and $V$ are unitary matrices and $\Sigma$ is a diagonal matrix with positive diagonal entries $\sigma_1 \geq\sigma_2 \geq .... \geq \sigma_r$ where $r$ is the rank of matrix $A$. The $\sigma_i$'s are called the singular values of $A$. We note that the first $r$ columns of $U$ will form an orthonormal basis for the column space of $A$ \cite{svd_prop}. Likewise, the first $r$ columns of $V$ will form an orthonormal basis for the row space of $A$. The orthonormal columns of $U$ and $V$ also contain the eigenvectors for the matrices $AA^*$ and $A^*A$ \cite{svd_prop}. This can be shown using the singular value decomposition of $A$ to get the following eigendecompositions:
\begin{enumerate}
    \item $A^*A = (U\Sigma V^*)^*(U\Sigma V^*) = V\Sigma^* U^*U\Sigma V^* = V\Sigma^*\Sigma V^* = V\Sigma^2 V^*.$
    \item $AA^* = (U\Sigma V^*)(U\Sigma V^*)^* = U\Sigma V^*V\Sigma^* U^* = U\Sigma\Sigma^* U^* = U\Sigma^2 U^*.$
\end{enumerate}
These properties of the singular value decomposition will become useful in Section \ref{section: eigenface} when we experiment with SVD through an eigenface example. 
\subsubsection{Randomized SVD}
Given a matrix $A$, we want to find a matrix $Q$ with orthonormal columns, such that $A \approx QQ^*A$ \cite{halko2009finding}.

The matrix $QQ^*$ is an orthogonal projector. A projector is a matrix that squares to itself. This means that applying it a second time to a given vector will do nothing because the vector has already been projected into the desired subspace. $QQ^*$ is a projector because
\begin{align*}
(QQ^*)^2 &= (QQ^*)(QQ^*) \\
&= Q(Q^*Q)Q^* \\
&= QIQ^* \\
&= QQ^* \,.
\end{align*}
It is an orthogonal projector because it is Hermitian (equal to its conjugate transpose). The kernel and row space of a matrix are orthogonal complements of each other. So, the kernel and column space are orthogonal iff the matrix is Hermitian.

We want an orthogonal projector primarily for two reasons. One reason is numerical stability -- the operator norm of an orthogonal projector is 1. Another reason is that it projects each vector to the closest possible vector in the subspace. Since it's not ``stretching'' vectors, distances are reasonably preserved. With a general projection, some vectors will be arbitrarily grown and others shrunk, depending on the specific projector (so it's not inherent to the data).

Using ideas from \cite{halko2009finding} we introduce randomness by constructing a $n\times k$ random Gauissan matrix $\Omega$. We set $Y = A\Omega$ and construct the matrix $Q$ whose columns for an orthonormal basis for $Y$. Then an approximate SVD can be computed as follows:

Let $B = Q^*A$. Then, we have $QB = QQ^*A \approx A$. We then compute the SVD of the small (relative to $A$) matrix $B$.

\begin{equation}
\label{randsvd}
B = \tilde{U} \Sigma V^*
\end{equation}

We take $U = Q \tilde{U}$, and we now have $A \approx U \Sigma V^*$. For this to be an exact SVD, we would need to have $U$ unitary, but since we are only trying to find a low-rank SVD approximation, it will in fact not be square, so the best we can do is ensure that it has orthonormal columns. This is equivalent to requiring $U^*U=I$. We have
\[U^*U = (Q\tilde{U})^*(Q\tilde{U}) = \tilde{U}^*Q^*Q\tilde{U}
= \tilde{U}^*\tilde{U} = I \,.\]

Note that traditionally in SVD, $U$ would need to be a square matrix, but here we have a rectangular matrix that contains only approximations to the most dominant singular vectors, not all of them.

Thus, finally, we have constructed a randomized low-rank approximation for the SVD of the matrix $A$.

\subsection{Interpolative Decomposition}

\subsubsection{Deterministic ID}
Given a matrix $A\in \mathbb{R}^{m\times n}$ we can come up with a low-rank matrix approximation that uses $A$'s own columns. As stated in \cite{id}, by reusing the columns of $A$, we are able to save space and keep the structure of the columns.

The interpolative decomposition can be computed using the column-pivoted $QR$ factorization:
\begin{equation}
    AP = QR
\end{equation}
where $P$ is a $n \times n$ permutation matrix moving picked columns to the front. The reordering of the columns of $A$ gives us a nice skeleton for the ID. Namely, the column-pivoted $QR$ chooses the ``best'' $k$ columns from $A$.

To obtain our low-rank approximation we form the submatrix $Q_k$ formed by the first $k$ columns of $Q$. Thus we have the approximation: 
\begin{equation}
    A \approx Q_k Q_k^*A
\end{equation}
which gives us a particular rank $k$ projection of $A$.

\subsubsection{Randomized ID} \label{section: rand_id}
We introduce a novel method to compute a randomized interpolative decomposition.

We randomly sample (without replacement) $p$ columns from the $n$ columns of $A$, where $p > k$. Let $A'$ denote the submatrix formed by these $p$ columns. We then perform a column-pivoted $QR$ factorization on $A'$:
\begin{equation}
    A'P = QR
\end{equation}
Similar to deterministic ID, we take the first $k$ columns of $Q$ to form the submatrix $Q_k$, giving us the decomposition
\begin{equation}
    A \approx Q_k Q_k^*A \,,
\end{equation}
where $Q_kQ_k^*A$ is a rank $k$ projection of A.
\subsection{Fixed-precision approximation problem}
Given a fixed approximation error $\varepsilon$ and a matrix $A$, we want to find a matrix $Q$ with orthonormal columns where $k = k(\varepsilon)$ such that: 
\begin{equation}
\| A-QQ^*A \| \leq \varepsilon
\end{equation}
In order for $A$ to be approximately equal to $QQ^*A$ the distance between the two matrices should be within the range of error $\varepsilon$.

Let $D=A-QQ^*A$. Since $Q^*A$ is a projection of the columns of $A$ onto a lower dimensional space, the Johnson-Lindenstrauss lemma guarantees that if $k>\frac{24}{3\eta^2-2\eta^3}\log n$ \cite{jl_lemma}, there exists such a $Q$ such that any row $D_i$ of $D$, $\|D_i\|_2^2<\eta$. If we set $\eta = \frac{\varepsilon^2}{n}$, and let $D$ denote $A-QQ^*A$, then
\[
\|A-QQ^*A\| = \sqrt{\sum_{i=1}^n \|D_i\|_2^2} \leq \sqrt{n\eta} = \varepsilon.
\]
Thus, a bound of $k>\frac{24n^3}{3\varepsilon^4n-2\varepsilon^6}\log n$ guarantees the existence of a $Q$ in order such that $\|A-QQ^*A\|<\varepsilon$ in the Frobenius norm.


\section{Kernel Methods}
\subsection{Deterministic Kernel Methods}

Kernel methods are ubiquitous in the fields of machine learning and statistics. These methods enable us to learn a nonlinear decision boundary using a linear classification algorithm. We do this by mapping the data from the low-dimensional input space into a high-dimensional feature space in which the data is linearly separable.

Since we only need to know the inner products between pairs of vectors in the feature space, we don't have to explicitly compute the feature map. This is much more computationally efficient. Letting $\phi$ denote the explicit high dimensional mapping, we need only compute \begin{equation}
    k(x,y)=\langle\phi(x),\phi(y)\rangle
\end{equation}
for each pair $(x,y)$ in the input space.

For many feature maps, there exist simple kernel matrices that we can use to perform easier computations:
\begin{itemize}
    \item Polynomial kernel
    \item Radial Basis Function (RBF) / Gaussian kernel
    \item etc.
\end{itemize}

\subsection{Kernel PCA}
Principal component analysis (PCA) is a common linear method for dimensionality reduction. Given a $n\times d$ data matrix $A$, the goal is to find a $n\times k$ representation, with $k<d$, that captures most of the information of the data. 

This can be done by column centering the data, labelling this as $A_0$, and computing an eigendecomposition of the covariance matrix
\begin{equation}
\frac{1}{n}A_0^TA_0 = Q\Lambda Q^{-1}
\end{equation}
Taking the first $k$ eigenvectors in $Q$ in order of decreasing eigenvalues yields the $k$ best \textit{principal components} of the data: an orthogonal set of $k$ linear combinations of the original features that captures the most variance in the data.

Often, when data is not linearly separable, we use kernel methods to project the data into a higher dimensional space before finding principal components. One trade off is that the principal components no longer represent explicit linear combinations of the original features, but rather linear combinations of the transformed features.

\subsection{Kernel SVM}
If we want to train a model on a set of labeled data, one option is to use a Support Vector Machine (SVM). If the data is linearly separable, this construct will find the $(d-1)$-dimensional hyperplane that best separates these $d$-dimensional points into their respective categories. In the simplest case, we have points in a plane, and we are separating them with a line.

When we say we want to find the 'best' separation, we mean that we want to find the separation that maximizes the minimum distance of the points to the hyperplane. This distance that we are trying to maximize is the margin. The intuition is that we want to have as clear of a separation between our two clusters of data points as possible.

If the data is not linearly separable, we can use the kernel trick to salvage the classification scheme. We project the data into a high-dimensional space, where the data is highly likely to be separable, and classify it in that feature space.

\subsection{Randomized Fourier Features}
\label{randomff}

In \cite{rahimirecht2008}, a randomized procedure for approximating the kernel is described by creating a low-dimensional map $z$ into $\Reals^m$ such that
\begin{equation}
k(x,y)=\langle \phi(x),\phi(y)\rangle \approx \frac{1}{m} z(x)z(y)^T.
\end{equation}
This can be done with the method of random Fourier features: given a shift-invariant real-valued kernel $k(x,y)$ on $\Reals^d\times\Reals^d$, if it is normalized such that $k(x,y)\leq 1$ for each $x,y$, then Bochner's theorem tells us that its Fourier transform $p(w)$ is a probability distribution. Then, we may approximate
\begin{align*}
    k(x,y) &= \int_{\Reals^d} p(w)e^{-\mathrm{j}w^T(x-y)}\mathrm{d}w \\
    &= \int_{\Reals^d} p(w)e^{-\mathrm{j}w^Tx}e^{\mathrm{j}w^Ty}\mathrm{d}w \\
    &\approx \frac{1}{m} \sum_{i=1}^m e^{-\mathrm{j}w_i^Tx}e^{\mathrm{j}w_i^Ty} \\
    &\approx \frac{1}{m} \sum_{i=1}^m \cos(w_i^Tx+b_i)\cos(w_i^Ty+b_i)
\end{align*}
where $w_i\sim p(w)$, $b_i\sim \text{Uniform}(0,2\pi)$. The first approximation is from Monte Carlo sampling to approximate the integral. For a given $m$, let 
\begin{equation}
z(x)=\sum_{i=1}^m \cos(w_i^Tx+b_i).
\end{equation}
to yield our approximation $\frac{1}{m}z(x)z(y)^T$.

As an example, consider a standard RBF kernel defined by
\begin{equation}
\label{rbfkern}
k(x,y) = \exp\left(-\gamma\|x-y\|_2^2 \right).
\end{equation}

We can approximate this kernel using $m$ random Fourier features as described above, with $w_i$ drawn from a multivariate normal distribution with mean $0$ and covariance $2\gamma I$.

Let $X\in\Reals^{n\times d}$ be our data matrix. Define the Kernel matrix $K\in\Reals^{n\times n}$ as $K_{ij}=k(x_i,x_j)$, and express our approximation $\hat{K}=\frac{1}{m}z(X)z(X)^T$ \cite{lopezpaz2014randomized}. Note that $\hat{K}$ is a rank $m$ approximation to $K$, and thus while these methods appear to be new, they are intimately connected to the randomized matrix decompositions earlier.

\subsection{Sampling over a range of parameters}
In some cases, an experimenter may wish to use the random Fourier features kernel approximation to approximate a parametric family of kernels, but may not know exactly what parameter choice to make. We introduce a novel method involving Monte Carlo sampling over a parametric range:

Let $k(x,y;\alpha)$ denote a real valued, normalized ($k(x,y;\alpha) \leq 1)$, shift-invariant parametric family of kernels  on $\mathbb{R}^d \times \mathbb{R}^d$, with parameters $\alpha\in E \subset \mathbb{R}^\ell$, where $E$ is the (Borel) parameter domain. Let $p(\alpha)$ be a probability distribution given by the inverse Fourier transform of $k$. For a given $m,q$, we may sample $\alpha_1,...,\alpha_{q}\sim\text{Uniform}(E)$ and subsequently $w_{s_1},\ldots,w_{s_{m}}\sim p(\alpha_s)$ for $s=1,\ldots,q$ and approximate the kernel, sampling over $E$:  
\begin{align*}  
    k(x,y) &= \int_E \int_{\mathbb{R}^d}p(\alpha)e^{-\mathrm{j}w^T(x-y)}\mathrm{d}w\mathrm{d}\alpha \\
    &\approx \frac{1}{q}\sum_{s=1}^{m} \int_{\mathbb{R}^d}p(w;\alpha_s)e^{-\mathrm{j}w_s^T(x-y)}\mathrm{d}w \\
    &\approx \frac{1}{mq} \sum_{s=1}^{q}\sum_{i=1}^{m} e^{-\mathrm{j}w_{s_i}^Tx}e^{\mathrm{j}w_{s_i}^Ty} \\
    &\approx \frac{1}{mq} \sum_{s=1}^{q}\sum_{i=1}^{m} \cos(w_{s_i}^Tx + b_i)\cos(w_{s_i}^Ty + b_i) \\ 
\end{align*}
where $b_i\sim\text{ Uniform}(0,2\pi)$.

This procedure may be useful in cases where efficiency is desired, and an optimal hyperparameter value is unknown, but instead a range is known. When the dataset is too large to test individual values in this range specifically (i.e. a grid search), this method may help to provide decent results at a low computational cost.

\section{Coding Investigations}\label{section: coding}
\subsection{Johnson-Lindenstrauss Lemma}
The code for this experiment can be found at
\url{https://rishi1999.github.io/random-projections/notebooks/html/JL_Lemma.html}

The Johnson-Lindenstrauss lemma is a powerful tool in dimension reduction. This lemma shows that when randomly projecting $n$ points in any dimension into a space of dimension $O(\log n)$ that pairwise distances are approximately preserved. In this section, we will provide experimental results to support one of the propositions instrumental to the proof of JL lemma from \cite{jl_lemma}:
\begin{prop}
Let $u\in \mathbb{R}^d$ be fixed, and let $R$ be a random matrix with $R_{ij} \sim N(0,1)$. Define $v=\frac{1}{\sqrt{k}}Ru$ such that $v\in\mathbb{R}^k$. Then
\begin{equation}
\mathbb{E} \,[\|v\|_2^2] = \|u\|_2^2
\end{equation}
\end{prop}
 This proposition is important as it allows us to randomly project a vector from a $d$-dimensional space into a $k$-dimensional space while preserving the squared Euclidean norm of the original vector in expectation. Algorithm \ref{JL_expectation} will allow us to test the proposition. It proceeds roughly as follows:
 
\begin{enumerate}
    \item Find the squared norm of a fixed high-dimensional vector
    \item Randomly project it 1000 times, and calculate the average squared norm of the projections
    \item Calculate the error between these two values
\end{enumerate}

\begin{lstlisting}[caption=JL lemma - error in random projections, label=JL_expectation, float=htb]
# create fixed unit vector u
u = random.randn(d,1)
u = u / np.linalg.norm(u)

# number of samples we will generate
iterations = 10000
v_errors = np.empty(iterations)

for i in range(iterations):
    # construct random Gaussian matrix
    R = random.randn(k,d)
    v = 1/math.sqrt(k) * R @ u
    # store squared 2-norm of v
    v_errors[i] = np.sum(np.square(v)) - 1

print(f'Mean: {np.mean(v_errors)}')
print(f'Stdev: {np.std(v_errors)}')
plt.hist(v_errors, bins=100)
\end{lstlisting}

To conduct this experiment we will let $u\in \mathbb{R}^{1000}$ and $v\in \mathbb{R}^{10}$. When we ran this algorithm, it computed an error of less than $0.01$. Figure \ref{jl_fig} shows that the relative error approximately centers around a mean value of 0. This shows that, in practice, the statement $\mathbb{E} \,[\|v\|_2^2] = \|u\|_2^2$ does hold when $u$ and $v$ are defined as in the above proposition.

\begin{figure}[htb]
        \centering
        \includegraphics[width = 8 cm]{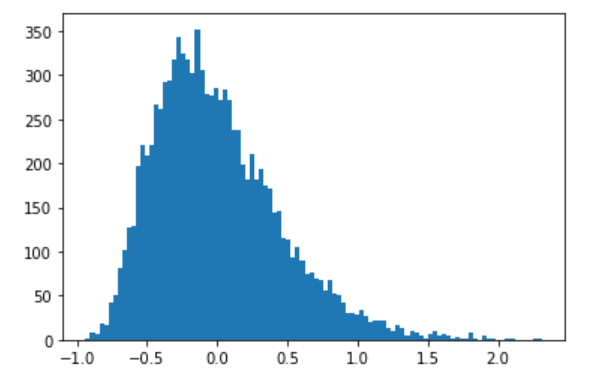}
        \caption{Error Between Squares Norms in High and Low Dimensional Spaces}
        \label{jl_fig}
\end{figure}

\FloatBarrier

\subsection{Random Decompositions}
The code for the following SVD/ID experiments can be found at
\url{https://rishi1999.github.io/random-projections/notebooks/html/Image_Compression.html}

Randomness is a valuable tool for performing low-rank matrix approximations. These efficient random methods for performing approximate matrix factorization enable us to process very large data sets at significantly lowered costs. Although random methods tend to be less accurate than deterministic methods, they can be much more efficient.  

In order to confirm that randomness does in fact improve low-rank approximations, we will experiment with two deterministic methods along with two random methods. We will then compare their relative errors and times by testing $620$ images from LFW dataset \cite{LFWTech} to form a $620\times 187500$ transpose matrix.  

\subsubsection{Interpolative Decomposition} \label{section: ID}
Given a matrix $A\in \mathbb{R}^{m \times n}$, we can compute an interpolative decomposition (ID), a low-rank matrix approximation that includes the original columns of $A$. One way we can do this is through the column-pivoted $QR$ factorization
\begin{equation}
    AP = QR \,,
\end{equation}
where $P$ is a permutation matrix. We take the first k columns from $Q$ to obtain the submatrix $Q_k$. Then we have the following low-rank decomposition:
\begin{equation}
    A \approx Q_k Q_k^\ast A \,.
\end{equation}

In Algorithm \ref{rand_id}, we use a new method described in Section \ref{section: rand_id} to compute a randomized ID (RID).
\begin{lstlisting}[caption=Randomized ID - Column Pivoted QR, label=rand_id, float=htb]
def random_id_rank_k(matrix, k, oversampling=10):
    p = k + oversampling
    m,n = A.shape
    cols = np.random.choice(n, replace=False, size=p)
    S = A[:,cols]
    q,r = np.linalg.qr(S,pivoting = True)
    q = q[:,:k]
    return q @ q.T @ A
\end{lstlisting}

Consider $d,m,r$ where $d$ is the deterministic matrix approximation, $m$ is the original data matrix, and $r$ is the randomized matrix approximation. We can then measure relative error for Figures \ref{rand_id_error_time} and \ref{rand_svd_error_time} in the following way:
\begin{enumerate}
    \item Compute absolute random error:  $ar = \|(r - m) \|_2$
    \item Compute absolute deterministic error: $ad = \|(d - m)\|_2$
    \item Calculate the error of $ar$ relative to $ad$: $relative.error = (ar-ad)/ad$
\end{enumerate}

Upon running the algorithm, as expected, the relative error for the RID tends to be higher than that of the ID. As we test the algorithm against higher values of $k$, we see in figure \ref{rand_id_error_time} that the random error does not decrease for larger rank $k$ approximations as quickly as the deterministic error. 

Despite the RID producing less accurate results, it is significantly more efficient. To show this the average time has been taken to test varying values of $k$ for both methods of computing the interpolative decomposition. In Figure \ref{rand_id_error_time}, it is shown that the random time relative to the deterministic time does not appear to be hardly growing at all as the value of $k$ increases. Accordingly, the randomized interpolative decomposition we have introduced here shows experimentally to be very computationally efficient. 

\begin{figure}[htb]
        \centering
        \includegraphics[width = 8 cm]{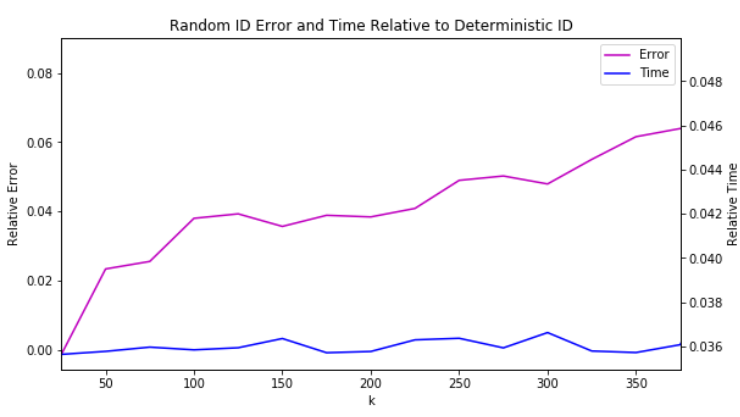}
        \caption{Random ID Error and Time Relative to Deterministic ID}
        \label{rand_id_error_time}
\end{figure}

\subsubsection{Singular Value Decomposition} \label{section:SVD}
Given a matrix $A\in \mathbb{R}^{m \times n}$ we can express the matrix as a product of three ``special'' matrices, the singular value decomposition (SVD):
\begin{align*}
    A = U_{m\times m}\Sigma_{m\times n} V^T_{n\times n}
\end{align*}
where $U$, $V$, and $\Sigma$ are the matrices defined in \ref{section: d_svd}.

We can compute a randomized SVD (RSVD) by first generating a random $n\times k$ matrix $\Omega$ \cite{halko2009finding}, and then forming the following $m\times k$ matrix $Y$: 
\begin{equation}
    Y = (AA^*)^q(A\Omega)
\end{equation}
where $q = 1$ or $q = 2$. In practice with $q = 0$, \cite{halko2009finding} tells us the algorithm can cause the singular spectrum of $A$ to decay slowly and thus the greatest singular values will not capture most of the variance.

Algorithm \ref{rand_svd} will allow us to test the accuracy and efficiency of this method using the following steps to compute the RSVD of $A$:
\begin{enumerate}
    \item Use $QR$ factorization to compute a matrix $Q$ whose orthonormal columns form a basis for the column space of $Y$.
    \item Set $B = Q^*A$
    \item Compute the SVD factorization such that: $B = U'\Sigma V^*$
    \item Thus $A \approx QQ^*A = QB = QU'\Sigma V^*$
\end{enumerate}
Note that we will be testing Algorithm \ref{rand_svd} using real matrices.

\begin{lstlisting}[caption=Randomized SVD, label=rand_svd, float=htb]
def random_svd_rank_k(A, k, power=1):
    omega = random.randn(A.shape[1],k)
    pow_matrix = np.linalg.matrix_power(A @ A.T,power)
    Y = pow_matrix @ (A @ omega)
    Q, R = np.linalg.qr(Y)
    B = Q.T @ A
    U_tilde, Sigma, Vh = np.linalg.svd(B)
    U = Q @ U_tilde
    Sigma = np.diag(Sigma)
    return U @ Sigma @ Vh[:k]
\end{lstlisting}

The results of running Algorithm \ref{rand_svd} for varying values of $k$ shows that the error for computing the RSVD is consistently slightly higher than computing the SVD. In Figure \ref{rand_svd_error_time}, we compare the absolute error of the RSVD relative to the absolute error of SVD as described in Section \ref{section: ID}. This graph also shows the average RSVD running time relative to the SVD running time. Figure \ref{rand_svd_error_time} shows us that the relative error is increasing. Although different from the ID, this is caused by our absolute error for SVD and RSVD decreasing at similar rates. Figure \ref{relative} demonstrates why the explanation for the increase in relative error for ID and SVD differs by showing the error for each method relative to the original data. As expected the RSVD method runs at a faster rate for smaller values of $k$ than SVD. However, it can be seen in Figure \ref{rand_svd_error_time} that as the values of $k$ increase, the RSVD algorithm is not only less accurate than the SVD algorithm, but less efficient as well.

\begin{figure}[htb]
        \centering
        \includegraphics[width = 8 cm]{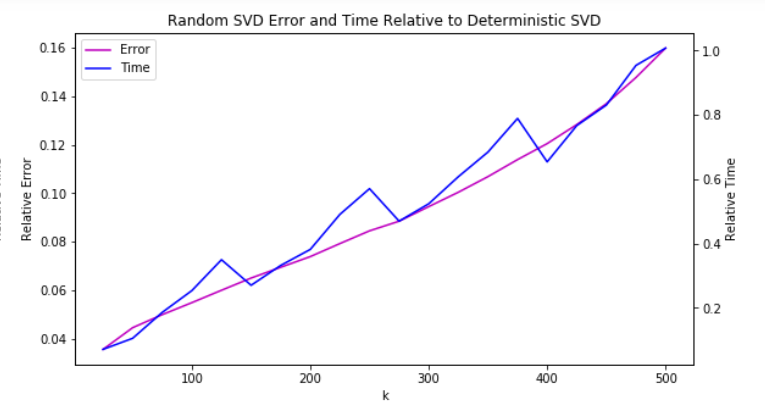}
        \caption{Random SVD Error and Time Relative to Deterministic SVD}
        \label{rand_svd_error_time}
\end{figure}

\begin{figure}[htb]
        \centering
        \includegraphics[width = 8 cm]{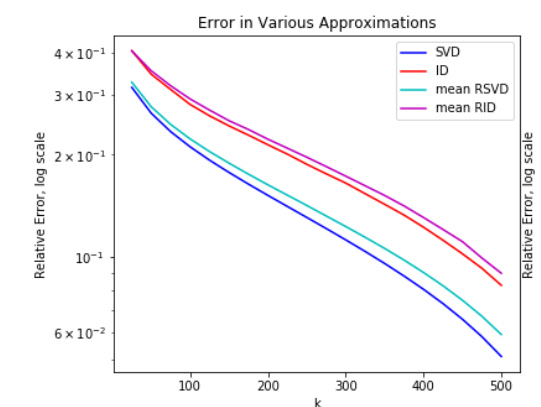}
        \caption{Error Relative to Original Data}
        \label{relative}
\end{figure}

From our experiments we see that, in general, SVD and RSVD have lower errors than ID and RID, and thus more accurate approximations. However, RSVD is far less computationally efficient than RID. Not only does does the randomized SVD lack efficiency for higher values of $k$, but our randomized ID is surprisingly just as efficient for smaller values of $k$ as it is for larger rank-$k$ approximations. Thus, when striving for efficiency or using large datasets, the RID is strongly preferred over RSVD. 

\FloatBarrier

\subsubsection{Eigenfaces} \label{section: eigenface}
The code for the following experiment can be found at
\url{https://rishi1999.github.io/random-projections/notebooks/html/Eigenfaces.html} 

One application of the SVD includes solving the eigenface problem. Using ideas from \cite{eigenfaces} our eigenfaces experiment tests the LFW dataset \cite{LFWTech}. This dataset contains more than 13,000 images of faces where each image is a $250\times 250$. By applying SVD to these images we can extract the most dominant features from each image, resulting in our set of eigenfaces. 

 Our algorithm starts with flattening each image to represent it as a vector of length $250\times 250\times 3 = 187500$. Note, we multiple by three to account for three colors channels of the images. In our experiment we will only use 620 images from the LFW dataset giving us a matrix $A$ of size $187500\times 620$. To normalize the data each column of the matrix will be subtracted by the mean face. This step allows us to take away the features that each face has in common, leaving each image with its distinctive features visible. Given $A$ with mean-subtracted columns, SVD can be performed. The eigenfaces of the data are then given by the columns of $U$. In our experiment we use both SVD and RSVD to compute the eigenfaces of $A$. 

\begin{figure}[htb]
        \centering
        \includegraphics[width = 8 cm]{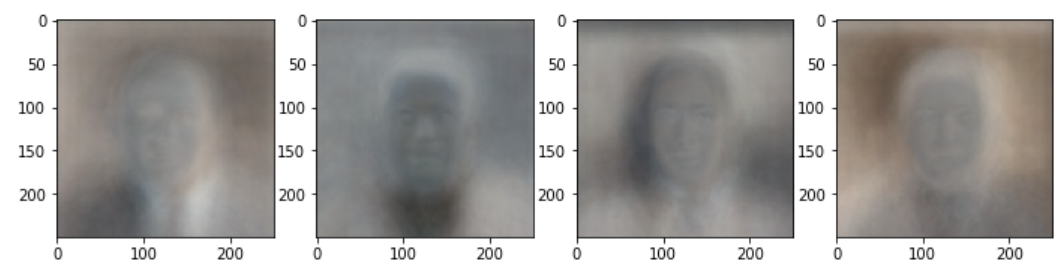}
        \caption{Eigenfaces obtained using Deterministic SVD}
        \label{det_basis}
\end{figure}
\begin{figure}[htb]
        \centering
        \includegraphics[width = 8 cm]{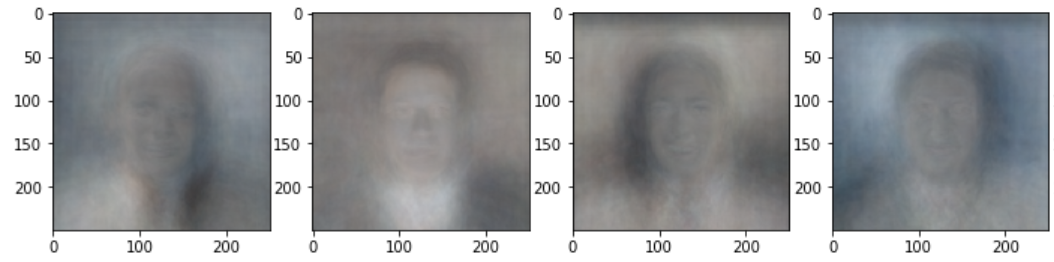}
        \caption{Eigenfaces obtained using Randomized SVD}
        \label{rand_basis}
\end{figure}

In Figure \ref{eigenface}, we display the absolute random error relative to the absolute deterministic error as well as the random time relative to deterministic time. As expected, given the experiment from Section \ref{section:SVD}, the relative error increases since the absolute errors for SVD and RSVD are decreasing at similar rates, which can be seen in Figure \ref{eigen_rel}. It can also be seen that, as the value of $k$, where $k$ is the number of columns of $U$, increases, the relative time increases until the randomized method is running at about the same speed as the deterministic method does. 
\begin{figure}[htb]
        \centering
        \includegraphics[width = 8 cm]{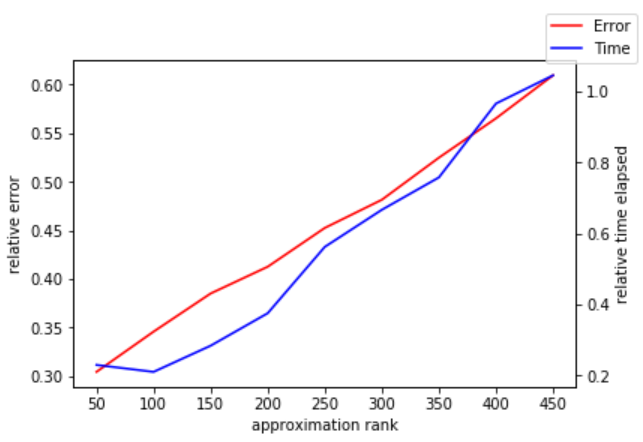}
        \caption{Random SVD Error and Time Relative to Deterministic SVD}
        \label{eigenface}
\end{figure}
\begin{figure}[htb]
        \centering
        \includegraphics[width = 8 cm]{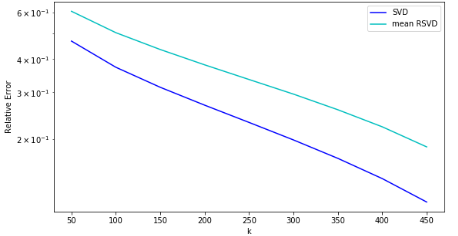}
        \caption{Error Relative to Original Data Matrix}
        \label{eigen_rel}
\end{figure}
\FloatBarrier

\subsection{Least-Squares Approximation}
The code for the following experiment can be found at
\url{https://rishi1999.github.io/random-projections/notebooks/html/Least_Squares.html}

When trying to solve the linear system of equation $Ax = b$ there is not always a vector $x$ that yields an exact solution. The solution can be approximated such that $\|Ax-b\|_2$ is minimized where $A$ is a full rank $m\times n$ matrix with $m \geq n$ and full column rank. To solve the least squares problem we have tested both a deterministic method that uses QR factorization and a randomized method. 

The deterministic method from \cite{qrnotes} used to calculate the linear least-squares problem solution utilizes $QR$ factorization to find a $x^*$ such that the equation $Ax = b$ is best approximated. Since $A$ has full column rank, $A$ has a unique $QR$ factorization: $A = QR$. Using the normal equations $A^TAx = A^Tb$ the vector $x$ can be approximated:
\begin{enumerate}
    \item $(QR)^TQRx = (QR)^Tb$
    \item $R^TQ^TQRx = R^TQ^Tb$ 
    \item Since $Q$ is an orthogonal matrix: $R^TRx = R^TQ^Tb$
    \item $R$ is an upper triangular matrix with positive diagonal entries. Thus $R$ has an inverse and so does its transpose. Thus the system can be solved so that: $x^* = R^{-1}Q^Tb$
\end{enumerate}
To test this method we use Algorithm \ref{det_ls} which generates a new $A$ matrix and $b$ vector each run where the dimensions of $A$ are increasing. 

\begin{lstlisting}[caption=Deterministic Least Squares Method, label=det_ls, float=htb]
dims = np.arange(100, 2000, step=50)
def ls(dims):
    times = []
    for n in tqdm(dims):
        m = 2 * n
        A = np.random.randn(m,n)
        b = np.random.randn(m,1)

        start = perf_counter()
        q,r = np.linalg.qr(A)
        qt = np.transpose(q)
        c = qt @ b
        rinv = np.linalg.inv(r)
        xls = rinv @ c
        end = perf_counter()
        
        times.append(end - start)
    return times
\end{lstlisting}

As expected, as the dimensions for $A$ increases so does the time it takes to run the algorithm. Figure \ref{det_ls_fig} shows that the absolute error is low for smaller matrices but continue to increase as the size of the matrix does. Overall, the absolute error for this method shows reasonably accurate results.

\begin{figure}[htb]
        \centering
        \includegraphics[width = 8 cm]{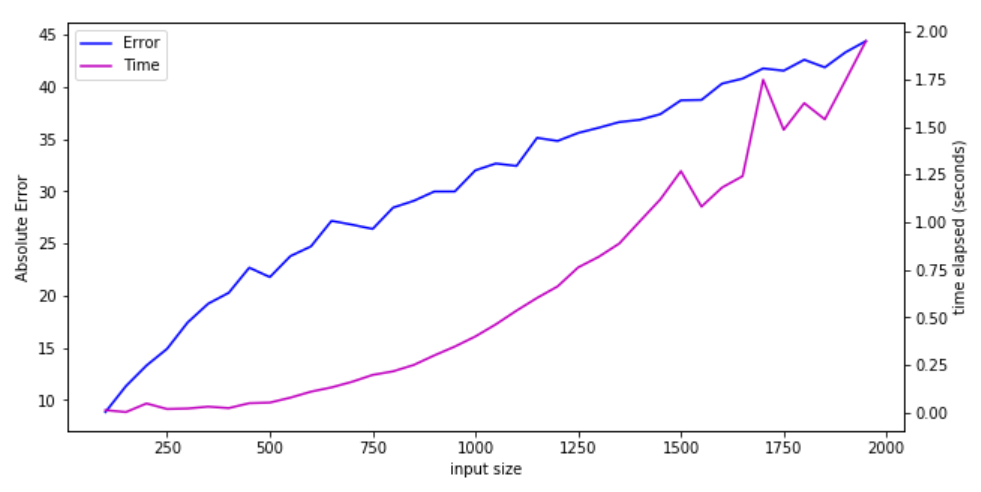}
        \caption{Efficiency and Error of Deterministic Least Squares Approximation Algorithm}
        \label{det_ls_fig}
\end{figure}

In an attempt to find a more efficient algorithm, we have created a random method that solves the least squares problem. Given an integer $k$, this method samples $k$ Gaussian vectors $x$ and keeps the vector that best minimizes $\lVert Ax-b \rVert_2$. This algorithm is then run on many random $A$ matrices and $b$ vectors with entries from a standard normal distribution. Unfortunately, this naive algorithm was unable to beat the deterministic one. In Figure \ref{rand_ls}, the random method proves not only to be less efficient but it is far less accurate.

\begin{figure}[htb]
        \centering
        \includegraphics[width = 8 cm]{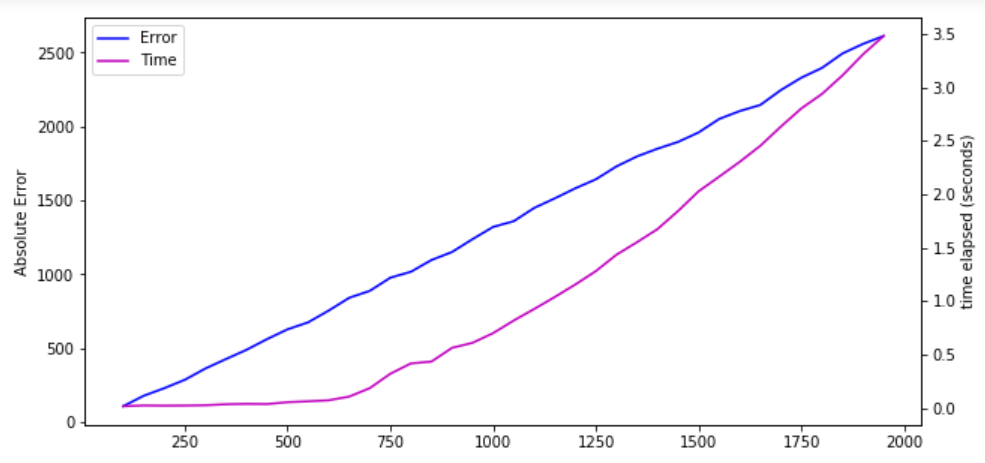}
        \caption{Efficiency and Error of Random Least Squares Approximation Algorithm}
        \label{rand_ls}
\end{figure}

Further investigations to fix this method may include finding a more structured way to randomly sample the $x$ vectors, instead of choosing completely arbitrarily from a standard distribution.

\FloatBarrier

\subsection{Randomized Kernel Methods}
In the following experiments we will use the randomized kernel method as described in Section \ref{randomff} to test $m\times d$ matrices. We provide pseduocode for our random kernel in Algorithm \ref{rand_kern}. 

\begin{lstlisting}[caption=Random Kernel Function, label=rand_kern, float=htb]
def generate_kernel(m=350, s=1/d):
    val = 2*np.pi
    b = np.random.uniform(low=0, high=val, size=(1,m))
    W = np.random.multivariate_normal(
            mean = np.zeros(d),
            cov = 2*s*np.eye(d),
            size = m
            ) #mxd
    def ker(x, y):
        z1 = np.cos(x @ W.T + b)
        z2 = np.cos(y @ W.T + b)
        return z1 @ z2.T / m
    return ker
\end{lstlisting}
\FloatBarrier
\subsubsection{Kernel PCA}
The code for the following KPCA experiment can be found at
\url{https://rishi1999.github.io/random-projections/notebooks/html/Kernel_PCA.html}

We began our investigation into the randomized Fourier features kernel approximation by applying it to principal component analysis (PCA). We investigated the effects of changing the hyperparameter $m$ on the resultant embedding for a conjured dataset of a circle surrounding a cloud of points.

In Figure \ref{rkpca_varying_gamma}, we display the embeddings yielded by plotting the projections onto the first two principal components preceded by a deterministic radial basis function (Gaussian) kernel.

In Figure \ref{rkpca_varying_m}, we vary $m$, the number of random Fourier features sampled, for each value of $\gamma$, the parameter of the Gaussian kernel seen in Equation \ref{rbfkern}. We observe how as $m$ grows, the embeddings more closely resemble their deterministic counterparts, shown in Figure \ref{rkpca_varying_gamma}. This conclusion is logical as $m$ solely represents the number of samples taken to approximate the integral (refer to Section \ref{randomff}). As we will see in Section \ref{randkersvmexp}, a low value of $m$ allows for lower computation cost, but with the trade-off of a worse approximation.

\begin{figure}[htb]
        \centering
        \includegraphics[width = 4cm]{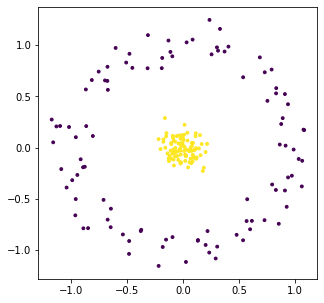}
        \caption{Original data}
        \label{rkpca_orig}
\end{figure}

\begin{figure}[htb]
        \centering
        \includegraphics[width = 10cm]{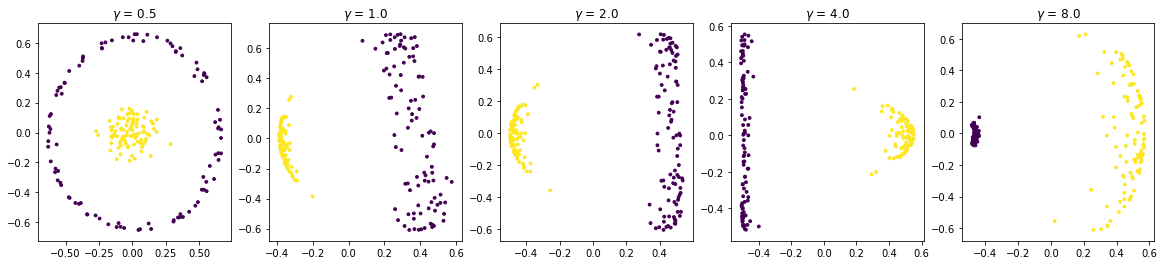}
        \caption{Embeddings with deterministic Gaussian kernel with varying $\gamma$ values}
        \label{rkpca_varying_gamma}
\end{figure}

\begin{figure}[htb]
        \centering
        \includegraphics[width = 12cm]{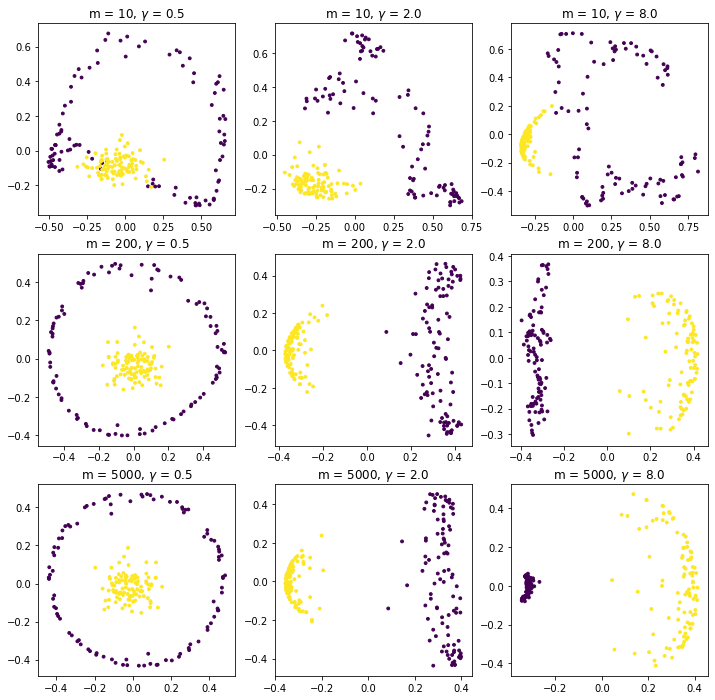}
        \caption{Embeddings with random Fourier features kernel approximating Gaussian kernel with varying $m$ values}
        \label{rkpca_varying_m}
\end{figure}

\FloatBarrier

\subsubsection{Kernel SVM}
The code for the following experiment along with other KSVM investigations can be found on the following pages:
\begin{itemize}
    \item \url{https://rishi1999.github.io/random-projections/notebooks/html/Kernel_SVM.html}
    \item \url{https://rishi1999.github.io/random-projections/notebooks/html/GridSearchSVM.html}
\end{itemize}

\label{randkersvmexp}
One experiment we ran was using the Kernel SVM technique to classify handwritten digits from the MNIST dataset \cite{lecun-mnisthandwrittendigit-2010}. Since this task is not binary classification (there are ten modes: one for each digit), we have to use a modified formulation of SVM to tackle the problem. By default, the scikit-learn \cite{scikit-learn} implementation of SVM uses a `one-vs-one' approach for multiclass classification; instead of performing a single instance of binary classification between two classes, we use the basic SVM to classify between each of the possible $10 \cdot 9/2=45$ pairs of classes and then tally up the results to determine which class fits best.

\begin{figure}[htb]
        \centering
        \includegraphics[width = 8 cm]{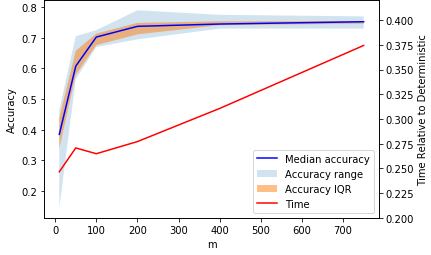}
        \caption{Randomized Kernel SVM Accuracy}
        \label{rksvm}
\end{figure}

In Figure \ref{rksvm}, we observe that as $m$, the number of random Fourier features sampled in the kernel approximation, grows, the accuracy grows (converging to the accuracy of the deterministic kernel), and the computational time increases as well. Thus, we see a similar trade-off between accuracy and time.

In addition, we tested the computational time needed to train and test (cross validate) SVMs on many different hyperparameter values, in the spirit of a grid search. Specifically, we performed three fold cross validation using a deterministic and randomized Gaussian kernel on sets of $100$ and $1000$ $\gamma$ values, and observed computational cost results in the following table.

\begin{table}[ht]
\begin{tabular}{llll}
Num. $\gamma$ values & Det. serial (s) & Rand. serial (s) & Rand. parallel (s) \\
100 & 133.03 & 78.97 & 41.18 \\
1000 & 1898.73 & 733.91 & 467.58
\end{tabular}
\end{table}

For the parallel column, we compute the kernel matrices
\begin{equation}
\hat{K} = \frac{1}{m}z(X)z(X)^T
\end{equation}
in parallel using \emph{batch matrix multiplication} rather than one at a time. We observe that testing hyperparameters using the randomized kernel is significantly faster than using the deterministic kernel, and that computing kernel matrices in parallel provides further speedup. In addition, this experiment is solely using $1000$ samples of MNIST, and the true power of the randomized kernel comes into play further when more samples are used.

For the parallelized approach, we note that the machine on which the experiments were run had two cores, pointing to an ideal speedup up $2.0$ (as ratio of serial to parallel). Our experimental speedup was $1.92$ for $100$ $\gamma$ values, and $1.57$ for $1000$ $\gamma$ values, showing a slight deviation from the optimal speedup.

In addition, we note for the $100$ $\gamma$ trial that the $\gamma$ value that produced the highest accuracy using the randomized kernel corresponded with the best deterministic $\gamma$ value, and for the $1000$ $\gamma$ value experiment, that the best random $\gamma$ corresponded with the 10th (up to uniqueness) best deterministic $\gamma$. These results show that when large amounts of parameters need to be tested, it can be efficient to test the results using the randomized method, find the best hyperparameter value, and use this value in the deterministic kernel.

\section{Conclusion}
Randomization is a powerful tool in low-rank matrix factorization and dimension reduction. 

Specifically, using randomness in matrix decompositions, despite losing accuracy, provides better efficiency. This provides a significant advantage as it enables us to deal with much larger datasets. In this paper we have discussed randomized methods for computing approximate matrix decompositions and compared these with their deterministic analogs. Our results show us that, in general, the SVD is more accurate then ID for both random and deterministic methods. However, the randomized SVD is far less efficient than random ID. As $k$ increases, where $k$ is the rank of the projection, the time it takes to run random SVD increases at a much greater rate then random ID. The results show that once the value of k is large enough the randomized SVD is not only less accurate then deterministic SVD but it is no longer more efficient. Based off our experiments, when aiming for efficiency or using large datasets, the randomized ID method is preferred to randomized SVD.

In the randomized Fourier features kernel approximation, we note that the randomized kernel is effectively a low-rank approximation of the deterministic kernel, with rank corresponding to the number of random Fourier features sampled. We note that the randomized kernel matrices were much less computationally costly to compute than their deterministic counterparts. Applications for these methods include using kernel methods such as PCA or SVM on large datasets, or when many hyperparameters values are to be tested, such as a grid search. In this case, the randomized kernel allows for parallelization using batch matrix multiplication. 

In conclusion, randomized methods are an excellent tool to use when efficiency is desired, especially in cases when their deterministic counterparts are computationally intractable. The phenomenon of concentration of measure allows the standard deviation of these methods to be surprisingly low, allowing for their usage in practical scenarios.

\nocite{icerm}

\emergencystretch=1em
\printbibliography[heading=bibintoc, title={References}]

\end{document}